\documentclass[preprint,12pt]{elsarticle}

 \newtheorem{thm}{Theorem}[section]

\newdefinition{rem}{Remark}[section]
 \newdefinition{example}{Example}
\newproof{proof}{Proof}

\usepackage{amsmath,amssymb}
\usepackage{varwidth}

\journal{Chaos, Solitons $\&$ Fractals  }

\begin{document}
\begin{frontmatter}
\title{$q$-Fibonacci bicomplex quaternions}

\author[rvt]{F.~Torunbalc{\i} Ayd{\i}n\corref{cor1}}
\ead{ftorunay@gmail.com}

\cortext[cor1]{Corresponding author}

 \address[rvt]{Yildiz Technical University,
Faculty of Chemical and Metallurgical Engineering,
Department of Mathematical Engineering,
Davutpasa Campus, 34220,
Esenler, \.{I}stanbul,  TURKEY}

  %

\begin{abstract}
In the paper, we define the $q$-Fibonacci bicomplex quaternions and the $q$-Lucas bicomplex quaternions, respectively. Then, we give some algebraic properties of $q$-Fibonacci bicomplex quaternions and the $q$-Lucas bicomplex quaternions.
\end{abstract}



\end{frontmatter}

\section{Introduction}
The real quaternions were first described by Irish mathematician William Rowan Hamilton in 1843. The real quaternions constitute an extension of complex numbers into a four-dimensional space and can be considered as four-dimensional vectors, in the same way that complex numbers are considered as two-dimensional vectors. Hamilton \cite{hamilton1866elements} introduced the set of quaternions which can be represented as     
\begin{equation}\label{1}
H=\left\{ \,q={{q}_{0}}+i\,{{q}_{1}}+j\,{{q}_{2}}+k\,{{q}_{3}}\,\left. {} \right|\ \ {{q}_{0}},\,{{q}_{1}},\,{{q}_{2}},\,{q}_{3}\in \mathbb R\, \right\}
\end{equation}
where
\begin{equation}\label{2}
{{i}^{2}}={{j}^{2}}={{k}^{2}}=-1\,,\ \ i\ j=-j\ i=k\,,\quad j\ k=-k \ j=i\,,\quad k\ i=-i\ k=j\,.
\end{equation}
Horadam \cite{horadam1963complex,horadam1993quaternion} defined complex Fibonacci and Lucas quaternions as follows

\begin{equation}\label{3}
Q_n=F_n+F_{n+1}\,i+F_{n+2}\,j+F_{n+3}\,k 
\end{equation}
and
\begin{equation}\label{4}
K_n=L_n+L_{n+1}\,i+L_{n+2}\,j+L_{n+3}\,k
\end{equation}
where $F_n$ and $L_n$ denote the $n-th$ Fibonacci and Lucas numbers, respectively. Also, the imaginary quaternion units $i\,j,\,\,k$ have the following rules
\begin{equation*}
i^2=j^2=k^2=-1\,,\,\, i\,j=-j\,i=k\,,\quad j\,k=-k \,j=i\,,\quad k\,i=-i\,k=j
\end{equation*}
There are several studies on quaternions as the Fibonacci and Lucas quaternions and their generalizations, for example,\cite{iyer1969some, iyer1969note, koshy2019fibonacci, vajda1989fibonacci, swamy1973generalized, halici2012fibonacci, halici2013complex, nurkan2015dual,clifford1882preliminary,yuce2016new, akkus2019quaternions}.\\
Bicomplex numbers were introduced by Corrado Segre in 1892 \cite{fulton2006corrado}. In 1991, G. Baley Price, the bicomplex numbers are given in his book based on multicomplex spaces and functions \cite{price1991introduction}. In recent years, fractal structures of these numbers have also been studied \cite{rochon2004algebraic, nurkan2015note}. The set of bicomplex numbers can be expressed by the basis $\{1\,,i\,,j\,,i\,j\,\}$ (Table 1) as, 
\begin{equation}\label{5}
\begin{aligned}
\mathbb{C}_2=\{\, q=q_1+i\,q_2+j\,q_3+i\,j\,q_4 \ | \, q_1,q_2,q_3,q_4\in \mathbb R\}
\end{aligned}
\end{equation}
where $i$,$j$ and $i\,j$ satisfy the conditions 
\begin{equation}\label{6}
i^2=-1,\,\,\,j^2=-1,\,\,\,i\,j=j\,i.  
\end{equation}\\
In 2018, the bicomplex Fibonacci quaternions defined by Ayd{\i}n Torunbalc{\i} \cite{aydin2018bicomplex} as follows
\begin{equation}\label{7}
\begin{aligned}
\mathbb{BF}_{n} & = {F}_{n}+i\,{F}_{n+1}+j\,{F}_{n+2}+i\,j\,{F}_{n+3} \\
& = ({F}_{n}+i\,{F}_{n+1})+({F}_{n+2}+i\,{F}_{n+3})\,j 
\end{aligned}
\end{equation}
where $i$, $j$ and $i\,j$ satisfy the conditions (\ref{6}).\\ 
The theory of the quantum $q$-calculus has been extensively studied in many branches of mathematics as well as in other areas in biology, physics, electrochemistry, economics, probability theory, and statistics \cite{arfken1999mathematical,adler1995quaternionic}. For  $n\in\mathbb{N}_{0}$ , the $q$-integer \, ${\lbrack{n}\rbrack}_{q}$ is defined as follows
\begin{equation}\label{8}
\begin{aligned}
{\lbrack{n}\rbrack}_{q} &= \frac{1-q^n}{1-q} = 1 + q + q^2 +\,\ldots\,+ q^{n-1}. 
\end{aligned}
\end{equation} 
By (\ref{8}), for all $m,n\in{\mathbb{Z}}$, can be easily obtained\,  $\lbrack{m + n}\rbrack_{q} = \lbrack{m}\rbrack_{q} + q^m \,\lbrack{n}\rbrack_{q}$. For more details related to the quantum $q$-calculus, we refer to \cite{andrews1999special,kac2002quantum}.\\
In 2019, $q$-Fibonacci hybrid and $q$-Lucas hybrid numbers defined by K{\i}z{\i}late\c{s} \cite{kizilatecs2020new} as follows
\begin{equation}\label{9}
\begin{aligned}
\mathbb{HF}_n(\alpha;q) =& \alpha^{n-1}\lbrack{n}\rbrack_{q}+\alpha^{n1}\lbrack{n+1}\rbrack_{q}\,\bold{i}+\alpha^{n+1}\lbrack{n+2}\rbrack_{q}\,\,\boldsymbol{\varepsilon}\\
&+\alpha^{n+2}\lbrack{n+3}\rbrack_{q}\,\bold{h}\, \\
\end{aligned}
\end{equation}
 and
\begin{equation}\label{10}
\begin{aligned}
\mathbb{HL}_n(\alpha;q) =& \alpha^{n}\frac{{\lbrack{2\,n}\rbrack}_{q}}{{\lbrack{n}\rbrack}_{q}}+\alpha^{n+1}\,\frac{{\lbrack{2\,n+2}\rbrack}_{q}}{{\lbrack{n+1}\rbrack}_{q}}\,\bold{i}+\alpha^{n+2}\,\frac{{\lbrack{2\,n+4}\rbrack}_{q}}{{\lbrack{n+2}\rbrack}_{q}}\,\,\boldsymbol{\varepsilon}\\
&+\alpha^{n+3}\frac{{\lbrack{2\,n+6}\rbrack}_{q}}{{\lbrack{n+3}\rbrack}_{q}}\,\bold{h} \\
\end{aligned}
\end{equation}
where $i$,\,\,$\varepsilon$ and $h$ satisfy the conditions 
\begin{equation}\label{11}
\bold{i^2}=-1,\,\,\,\boldsymbol{\varepsilon^2}=0,\,\,\,\bold{h^2}=1,\,\,\bold{i}\,\bold{h}=\bold{h}\,\bold{i}=\boldsymbol{\varepsilon}+\bold{i}.  
\end{equation} \,
Also, K{\i}z{\i}late\c{s} derived several interesting properties of these numbers such as Binet-Like formulas, exponential generating functions, summation formulas, Cassini-like identities, Catalan-like identities and d’Ocagne-like identities \cite{kizilatecs2020new}.\\
 
\section{$q$-Fibonacci bicomplex quaternions}

In this section, we define $q$-Fibonacci bicomplex quaternions and $q$-Lucas bicomplex quaternions by using the basis $\{1,\,i\,,j\,,i\,j\}$, where $i$,\,$j$ \,and\, $i\,j$ satisfy the conditions (\ref{6}) as follows
\begin{equation}\label{12}
\begin{array}{rl}
\mathbb{BF}_n(\alpha;q) =&\alpha^{n-1}\lbrack{n}\rbrack_{q}+\alpha^{n}\lbrack{n+1}\rbrack_{q}\,\,i+\alpha^{n+1}\lbrack{n+2}\rbrack_{q}\,\,j+\alpha^{n+2}\lbrack{n+3}\rbrack_{q}\,i\,j \\
\\
=&\alpha^{n}\,(\frac{1-q^n}{\alpha-\alpha\,q})+\alpha^{n+1}\,(\frac{1-q^{n+1}}{\alpha-\alpha\,q})\,i \\
\\
&+\alpha^{n+2}\,(\frac{1-q^{n+2}}{\alpha-\alpha\,q})\,j+\alpha^{n+3}\,(\frac{1-q^{n+3}}{\alpha-\alpha\,q})\,\,i\,j\\
\\
=&\frac{\alpha^{n}}{\alpha -(\alpha\,q)}\,[\,1+\alpha\,i+\alpha^2\,j+\alpha^3\,i\,j\,] \\
\\
&-\frac{(\alpha\,q)^{n}}{\alpha -(\alpha\,q)}\,\,[\,1+(\alpha\,q)\,i+(\alpha\,q)^2\,j+(\alpha\,q)^3\,i\,j\,]  
\end{array}
\end{equation}

and
\begin{equation}\label{13}
\begin{array}{rl}
\mathbb{BL}_n(\alpha;q) =&\alpha^{n}\frac{{\lbrack{2\,n}\rbrack}_{q}}{{\lbrack{n}\rbrack}_{q}}+\alpha^{n+1}\,\frac{{\lbrack{2\,n+2}\rbrack}_{q}}{{\lbrack{n+1}\rbrack}_{q}}\,i+\alpha^{n+2}\,\frac{{\lbrack{2\,n+4}\rbrack}_{q}}{{\lbrack{n+2}\rbrack}_{q}}\,j
+\alpha^{n+3}\frac{{\lbrack{2\,n+6}\rbrack}_{q}}{{\lbrack{n+3}\rbrack}_{q}}\,i\,j\\
\\
=&\alpha^{2n}\,(\frac{1-q^{2n}}{\alpha^n-(\alpha\,q)^n})+\alpha^{2n+2}\,(\frac{1-q^{2n+2}}{\alpha^{n+1}-(\alpha\,q)^{n+1}})\,i \\
\\
&+\alpha^{2n+4}\,(\frac{1-q^{2n+4}}{\alpha^{n+2}-(\alpha\,q)^{n+2}})\,j+\alpha^{2n+6}\,(\frac{1-q^{2n+6}}{\alpha^{n+3}-(\alpha\,q)^{n+3}})\,i\,j \\
\\
=&{\alpha^{n}}\,(1+\alpha\,i+\alpha^2\,j+\alpha^3\,i\,j\,)\\
\\
&-(\alpha\,q)^{n}\,(1+(\alpha\,q)\,i+(\alpha\,q)^2\,j+(\alpha\,q)^3\,i\,j\,) 
\end{array}
\end{equation}
For $\alpha=\frac{1+\sqrt{5}}{2}$ and \,$(\alpha\,q)=\frac{-1}{\alpha}$,\,\,\,{q}-Fibonacci bicomplex quaternion \,$\mathbb{BF}_n(\alpha;q)$\,\,become the bicomplex Fibonacci quaternions $\mathbb{BF}_n$.\\
\\
The addition, substraction and multiplication by real scalars of two $q$-Fibonacci bicomplex quaternions gives $q$-Fibonacci bicomplex quaternion.\\  
Then, the addition, subtraction and multiplication by scalar of $q$-Fibonacci bicomplex quaternions are defined by 
\begin{equation}\label{14}
\begin{array}{rl}
\mathbb{BF}_n(\alpha;q)\pm\mathbb{BF}_m(\alpha;q)=&(\alpha^{n-1}\lbrack{n}\rbrack_{q}+\alpha^{n}\lbrack{n+1}\rbrack_{q}\,i+\alpha^{n+1}\lbrack{n+2}\rbrack_{q}\,j\\
&+\alpha^{n+2}\lbrack{n+3}\rbrack_{q}\,i\,j) \\
&\pm(\alpha^{m-1}\lbrack{m}\rbrack_{q}+\alpha^{m}\lbrack{m+1}\rbrack_{q}\,i+\alpha^{m+1}\lbrack{m+2}\rbrack_{q}\,j\\
&+\alpha^{m+2}\lbrack{m+3}\rbrack_{q}\,i\,j) \\
=&[\alpha^{n}(\frac{1-q^n}{\alpha-\alpha\,q})\pm\alpha^{m}(\frac{1-q^m}{\alpha-\alpha\,q})] \\
&+[\alpha^{n+1}(\frac{1-q^{n+1}}{\alpha-\alpha\,q})\pm\alpha^{m+1}(\frac{1-q^{m+1}}{\alpha-\alpha\,q})]\,i \\
&+[\alpha^{n+2}(\frac{1-q^{n+2}}{\alpha-\alpha\,q})\pm\alpha^{m+2}(\frac{1-q^{m+2}}{\alpha-\alpha\,q})]\,j \\
&+[\alpha^{n+3}(\frac{1-q^{n+3}}{\alpha-\alpha\,q})\pm\alpha^{m+3}(\frac{1-q^{m+3}}{\alpha-\alpha\,q})]\,i\,j \\
=&\frac{1}{\alpha-\alpha\,q}\,\{\,(\alpha^{n}\pm\alpha^{m})(1+\alpha\,i+\alpha^2\,j+\alpha^3\,i\,j) \\
&+((\alpha\,q)^n\pm(\alpha\,q)^m\,)(1+(\alpha\,q)\,i+(\alpha\,q)^2\,j\\
&+(\alpha\,q)^3\,i\,j\,)\,\}.
\end{array}
\end{equation}
The multiplication of $q$-Fibonacci bicomplex quaternion by the real scalar $\lambda$ is defined as 
\begin{equation}\label{15}
\begin{array}{rl}
{\lambda}\,\mathbb{BF}_n(\alpha;q)&=\lambda\,\alpha^{n}(1+\alpha\,i+\alpha^2\,j+\alpha^3\,i\,j\,)\\
&+\lambda\,(\alpha\,q)^n\,(\,1+(\alpha\,q)\,i+(\alpha\,q)^2\,j+(\alpha\,q)^3\,i\,j\,).\\
\end{array}
\end{equation}
\\
The scalar  and the vector part of \, $\mathbb{BF}_n(\alpha;q)$ which is the $n-th$ term of the $q$-Fibonacci bicomplex quaternion are denoted by 
\begin{equation}\label{17}
{S}_{\mathbb{BF}_n(\alpha;q)}=\alpha^{n-1}\lbrack{n}\rbrack_{q},\,\,
{V}_{\mathbb{BF}_n(\alpha;q)}=\alpha^{n}\lbrack{n+1}\rbrack_{q}\,i+\alpha^{n+1}\lbrack{n+2}\rbrack_{q}\,j+\alpha^{n+2}\lbrack{n+3}\rbrack_{q}\,i\,j.
\end{equation}
\\
Thus, the $q$-Fibonacci bicomplex quaternion $\mathbb{BF}_n(\alpha;q)$  is given by \\ $$\mathbb{BF}_n(\alpha;q)={S}_{\mathbb{BF}_n(\alpha;q)}+{V}_{\mathbb{BF}_n(\alpha;q)}$$.
\\ 
The multiplication of two $q$-Fibonacci bicomplex quaternions is defined by
\begin{equation}\label{16}
\begin{array}{rl}
\mathbb{BF}_n(\alpha;q)\times\,\mathbb{BF}_m(\alpha;q)=&(\alpha^{n-1}\lbrack{n}\rbrack_{q}+\alpha^{n}\lbrack{n+1}\rbrack_{q}\,i+\alpha^{n+1}\lbrack{n+2}\rbrack_{q}\,j\\
&+\alpha^{n+2}\lbrack{n+3}\rbrack_{q}\,i\,j)\\
&\times\,(\alpha^{m-1}\lbrack{m}\rbrack_{q}+\alpha^{m}\lbrack{m+1}\rbrack_{q}\,i+\alpha^{m+1}\lbrack{m+2}\rbrack_{q}\,j\\
&+\alpha^{m+2}\lbrack{m+3}\rbrack_{q}\,i\,j) \\
\\
=&\frac{1}{\alpha-\alpha\,q}\,(\alpha^{n+m})\,\{\,(1-\alpha^2-\alpha^4+\alpha^6)\\
&+2\,i\,(\alpha-\alpha^5)+2\,j\,(\alpha^2-\alpha^4)+4\,i\,j\,(\alpha^3)\}\\
&-q^m\,\{(1-\alpha(\alpha\,q)-\alpha^2(\alpha\,q)^2+\alpha^3(\alpha\,q)^3)\\
&+i\,(\alpha+(\alpha\,q)-\alpha^2(\alpha\,q)^3-\alpha^3(\alpha\,q)^2)\\
&+j\,(\alpha^2+(\alpha\,q)^2-\alpha(\alpha\,q)^3-\alpha^3(\alpha\,q))\\
&+i\,j\,(\alpha^3+(\alpha\,q)^3-\alpha(\alpha\,q)^2-\alpha^2(\alpha\,q))\}\\
&-q^n\,\{(1-\alpha(\alpha\,q)-\alpha^2(\alpha\,q)^2+\alpha^3(\alpha\,q)^3)\\
&+i\,(\alpha+(\alpha\,q)-\alpha^2(\alpha\,q)^3-\alpha^3(\alpha\,q)^2)\\
&+j\,(\alpha^2+(\alpha\,q)^2-\alpha(\alpha\,q)^3-\alpha^3(\alpha\,q))\\
&+i\,j\,(\alpha^3+(\alpha\,q)^3+\alpha^2(\alpha\,q)+\alpha(\alpha\,q)^2)\}\\
&+q^{n+m}\,\{(1-(\alpha\,q)^2-(\alpha\,q)^4+(\alpha\,q)^6)\\
&+2\,i\,((\alpha\,q)-(\alpha\,q)^5)+2\,j\,((\alpha\,q)^2-(\alpha\,q)^4)\\
&+4\,i\,j\,((\alpha\,q)^3\,)\}\\
=&\mathbb{BF}_m(\alpha;q)\times\,\mathbb{BF}_n(\alpha;q)
\end{array}
\end{equation} 
Here, quaternion multiplication is done using bicomplex quaternionic units (table 1), and this product is commutative.
\\
\begin{table}[]
\centering
\caption{Multiplication scheme of bicomplex units}
\begin{tabular}{c rrrr}
\hline
 $x$&  $1$&  $i$& $j$& $i\,j$\\  
\hline
 $1$&  $1$&  $i$&  $j$& $i\,j$\\  
 $i$&  $i$&  $-1$&  $i\,j$& $-j$\\ 
$j$&  $j$&  $i\,j$&  $-1$& $-i$\\
$i\,j$& $i\,j$& $-j$& $-i$& $1$\\ 
\hline 
\end{tabular}
\end{table}
\\
Also, the $q$-Fibonacci bicomplex quaternion product may be obtained as follows \\
\begin{equation}\label{18}
\begin{array}{lr}
\mathbb{BF}_n(\alpha;q)\times\,\mathbb{BF}_m(\alpha;q)=\\
\\
\scriptsize{
\begin{matrix}
&\left(\begin{array}{cccc}  
\alpha^{n-1}\lbrack{n}\rbrack_{q} & -\alpha^{n}\lbrack{n+1}\rbrack_{q} & -\alpha^{n+1}\lbrack{n+2}\rbrack_{q} & \alpha^{n+2}\lbrack{n+3}\rbrack_{q}  \\ 
\alpha^{n}\lbrack{n+1}\rbrack_{q} & -\alpha^{n-1}\lbrack{n}\rbrack_{q} & -\alpha^{n+2}\lbrack{n+3}\rbrack_{q} & -\alpha^{n+1}\lbrack{n+2}\rbrack_{q}  \\
\alpha^{n+1}\lbrack{n+2}\rbrack_{q} & -\alpha^{n+2}\lbrack{n+3}\rbrack_{q} & \alpha^{n-1}\lbrack{n}\rbrack_{q} & -\alpha^{n}\lbrack{n+1}\rbrack_{q}  \\ 
\alpha^{n+2}\lbrack{n+3}\rbrack_{q} & \alpha^{n+1}\lbrack{n+2}\rbrack_{q} & \alpha^{n}\lbrack{n+1}\rbrack_{q} & \alpha^{n-1}\lbrack{n}\rbrack_{q}  
\end{array} \right)
\end{matrix}}
.
\scriptsize{
\begin{matrix} 
\left(\begin{array}{c}
\alpha^{m-1}\lbrack{m}\rbrack_{q} \\ 
\alpha^{m}\lbrack{m+1}\rbrack_{q}  \\
\alpha^{m+1}\lbrack{m+2}\rbrack_{q} \\ 
\alpha^{m+2}\lbrack{m+3}\rbrack_{q}  
\end{array} \right)
\end{matrix}}
\end{array}
\end{equation}

\medskip

Three kinds of conjugation can be defined for bicomplex numbers \cite{rochon2004algebraic, nurkan2015note}. Therefore, conjugation of the $q$-Fibonacci bicomplex quaternion is defined in three different ways as follows
\begin{equation} \label{19}
\begin{aligned}
(\mathbb{BF}_n(\alpha;q))^{*_1}=&(\alpha^{n-1}\lbrack{n}\rbrack_{q}-i\,\alpha^{n}\lbrack{n+1}\rbrack_{q}+j\,\alpha^{n+1}\lbrack{n+2}\rbrack_{q}-i\,j\,\alpha^{n+2}\lbrack{n+3}\rbrack_{q}), \\
\end{aligned}
\end{equation}
\begin{equation} \label{20}
\begin{aligned}
(\mathbb{BF}_n(\alpha;q))^{*_2}=&(\alpha^{n-1}\lbrack{n}\rbrack_{q}+i\,\alpha^{n}\lbrack{n+1}\rbrack_{q}-j\,\alpha^{n+1}\lbrack{n+2}\rbrack_{q}-i\,j\,\alpha^{n+2}\lbrack{n+3}\rbrack_{q}), \\ 
\end{aligned}
\end{equation}
\begin{equation} \label{21}
\begin{aligned}
(\mathbb{BF}_n(\alpha;q))^{*_3}=&(\alpha^{n-1}\lbrack{n}\rbrack_{q}-i\,\alpha^{n}\lbrack{n+1}\rbrack_{q}-j\,\alpha^{n+1}\lbrack{n+2}\rbrack_{q}+i\,j\,\alpha^{n+2}\lbrack{n+3}\rbrack_{q}). 
\end{aligned}
\end{equation}
\\
Therefore, the norm of the $q$-Fibonacci bicomplex quaternion ${\,\mathbb{BF}_n(\alpha;q)}$ is defined in three different ways as follows
\begin{equation}\label{22}
\begin{array}{rl}
{N}_(\mathbb{BF}_n(\alpha;q))^{*_1}=&\|(\mathbb{BF}_n(\alpha;q))\times\,(\mathbb{BF}_n(\alpha;q))^{*_1}\|^2, 
\end{array} 
\end{equation}
\begin{equation}\label{23}
\begin{array}{rl}
{N}_(\mathbb{BF}_n(\alpha;q))^{*_2}=&\|(\mathbb{BF}_n(\alpha;q))\times\,(\mathbb{BF}_n(\alpha;q))^{*_2}\|^2, 
\end{array} 
\end{equation}
\begin{equation}\label{24}
\begin{array}{rl}
{N}_(\mathbb{BF}_n(\alpha;q))^{*_3}=&\|(\mathbb{BF}_n(\alpha;q))\times\,(\mathbb{BF}_n(\alpha;q))^{*_3}\|^2. 
\end{array}  
\end{equation}
\\
\begin{thm} \textbf{(Binet's Formula)}. Let ${\mathbb{BF}_n(\alpha;q)}$ and ${\mathbb{BL}_n(\alpha;q)}$be the $q$-Fibonacci bicomplex quaternion and the $q$-Lucas bicomplex quaternion. For $n\ge 1$, Binet's formula for these quaternions respectively, is as follows:
\begin{equation}\label{25}
\mathbb{BF}_n(\alpha;q)=\frac{\alpha^n\,\widehat{\gamma}-(\alpha\,q)^n\,\widehat{\delta}}{\alpha-\alpha\,q},
\end{equation}
and
\begin{equation}\label{26}
\mathbb{BL}_n(\alpha;q)={\alpha^n\,\widehat{\gamma}+(\alpha\,q)^n\,\widehat{\delta}}
\end{equation}
where
\begin{equation*}
\begin{array}{l}
\widehat{\gamma }=1+{\alpha}\,i+{\alpha}^2\,j+{\alpha}^3\,i\,j,\,\,\,\,\, \alpha=\frac{1+\sqrt{5}}{2}
\end{array}
\end{equation*}
and
\begin{equation*}
\begin{array}{l}
\widehat{\delta }=1+(\alpha\,q)\,i+(\alpha\,q)^2\,j+(\alpha\,q)^3\,i\,j,\,\,\,\,\, \alpha\,q=\frac{-1}{\alpha}.
\end{array}
\end{equation*}
\end{thm}
\begin{proof}
(\ref{25}): Using (\ref{8}) and  (\ref{12}), we find that 
\begin{equation*}
\begin{array}{rl}
\mathbb{BF}_n(\alpha;q)=&\alpha^{n-1}\lbrack{n}\rbrack_{q}+\alpha^{n}\lbrack{n+1}\rbrack_{q}\,i+\alpha^{n+1}\lbrack{n+2}\rbrack_{q}\,j+\alpha^{n+2}\lbrack{n+3}\rbrack_{q}\,i\,j\, \\
\\
=&\alpha^n\,\frac{1-q^n}{\alpha -\alpha\,q}+\alpha^{n+1}\,\frac{1-q^{n+1}}{\alpha -\alpha\,q}\,i+\alpha^{n+2}\,\frac{1-q^{n+2}}{\alpha -\alpha\,q}\,j+\alpha^{n+3}\,\frac{1-q^{n+3}}{\alpha -\alpha\,q}\,i\,j \\
\\
=&\frac{\alpha^{n}\,[\,1+\alpha\,i+\alpha^2\,j+\alpha^3\,i\,j\,]-(\alpha\,q)^{n}\,[\,1+(\alpha\,q)\,i+(\alpha\,q)^2\,j+(\alpha\,q)^3\,i\,j\,]}{\alpha -(\alpha\,q)} \\
\\
=&\frac{\alpha^n\,\widehat{\gamma}-(\alpha\,q)^n\,\widehat{\delta}}{\alpha-\alpha\,q}
\end{array}
\end{equation*}
\\
In a similar way, equality (\ref{26}) can be derived as follows
\begin{equation*}
\begin{array}{rl}
\mathbb{BL}_n(\alpha;q) =&\alpha^{n}\frac{{\lbrack{2\,n}\rbrack}_{q}}{{\lbrack{n}\rbrack}_{q}}+\alpha^{n+1}\,\frac{{\lbrack{2\,n+2}\rbrack}_{q}}{{\lbrack{n+1}\rbrack}_{q}}\,i+\alpha^{n+2}\,\frac{{\lbrack{2\,n+4}\rbrack}_{q}}{{\lbrack{n+2}\rbrack}_{q}}\,j
+\alpha^{n+3}\frac{{\lbrack{2\,n+6}\rbrack}_{q}}{{\lbrack{n+3}\rbrack}_{q}}\,i\,j\\
=&\alpha^{2n}\,(\frac{1-q^{2n}}{\alpha^n-(\alpha\,q)^n})+\alpha^{2n+2}\,(\frac{1-q^{2n+2}}{\alpha^{n+1}-(\alpha\,q)^{n+1}})\,i \\
\\
&+\alpha^{2n+4}\,(\frac{1-q^{2n+4}}{\alpha^{n+2}-(\alpha\,q)^{n+2}})\,j+\alpha^{2n+6}\,(\frac{1-q^{2n+6}}{\alpha^{n+3}-(\alpha\,q)^{n+3}})\,i\,j \\ 
=&\frac{\alpha^{2n}}{\alpha^n}\,(1+\alpha\,i+\alpha^2\,j+\alpha^3\,i\,j\,)\\
\\
&-\frac{(\alpha\,q)^{2n}}{(\alpha\,q)^n}\,(1+(\alpha\,q)\,i+(\alpha\,q)^2\,j+(\alpha\,q)^3\,i\,j\,)\\
\\
=&\alpha^n\,\widehat{\gamma}-(\alpha\,q)^n\,\widehat{\delta}.
\end{array}
\end{equation*}
where \, $\widehat{\gamma }=1+{\alpha}\,i+{\alpha}^2\,j+{\alpha}^3\,i\,j$, \, \, $\widehat{\delta }=1+(\alpha\,q)\,i+(\alpha\,q)^2\,j+(\alpha\,q)^3\,i\,j$  and   $\widehat{\gamma }\,\widehat{\delta}=\widehat{\delta}\,\widehat{\gamma}$.\\
\end{proof}

\medskip

\begin{thm} \textbf{(Exponential generating function)} \\
Let $\mathbb{BF}_n(\alpha;q)$ be the $q$-Fibonacci bicomplex quaternion. For the exponential generating function for these quaternions is as follows:
\begin{equation}\label{27}
\begin{aligned}
g_{\mathbb{BF}_n(\alpha;q)}\,(\frac{t^n}{n!})=&\sum\limits_{n=0}^{\infty}\,{\mathbb{BF}_n(\alpha;q)}\,\frac{t^n}{n!}=\frac{\widehat{\gamma}\,e^{\alpha\,t}\,-\widehat{\delta}\,e^{(\alpha\,q)\,t}}{\alpha-\alpha\,q}\,
\end{aligned}
\end{equation}
\end{thm}
\begin{proof}
Using the definition of exponential generating function, we obtain
\begin{equation}\label{28}
\begin{array}{rl}
\sum\limits_{n=0}^{\infty}\,{\mathbb{BF}_n(\alpha;q)}\,\frac{t^{n}}{n!}&=\sum\limits_{n=0}^{\infty}\,(\frac{\alpha^n\,\widehat{\gamma}-(\alpha\,q)^n\,\widehat{\delta}}{\alpha-\alpha\,q})\,\frac{t^n}{n!}\\
&=\frac{\widehat{\gamma}}{\alpha-\alpha\,q}\,\sum\limits_{n=0}^{\infty}\,\frac{t^{n}}{n!}-\frac{\widehat{\delta}}{\alpha-\alpha\,q}\,\sum\limits_{n=0}^{\infty}\,\frac{(\alpha\,q\,t)^{n}}{n!}\\
&=\frac{\widehat{\gamma}\,e^{\alpha\,t}\,-\widehat{\delta}\,e^{(\alpha\,q)\,t}}{\alpha-\alpha\,q}.
\end{array}
\end{equation}
Thus, the proof is completed.
\end{proof}

\medskip
 
\begin{thm} \textbf{(Honsberger identity)} \\ 
For \,$n,m\ge 0$ the Honsberger identity for  the $q$-Fibonacci bicomplex quaternions ${\mathbb{BF}_n(\alpha;q)}$ and ${\mathbb{BF}_m(\alpha;q)}$ \, is given by
\begin{equation}\label{29}
\begin{array}{lr}
\mathbb{BF}_n(\alpha;q)\,\mathbb{BF}_m(\alpha;q)+\mathbb{BF}_{n+1}(\alpha;q)\,\mathbb{BF}_{m+1}(\alpha;q)\\
\\
=\frac{\alpha^{n+m}}{\alpha-\alpha\,q}\,\{\,(1+\alpha^2)\,\widehat{\gamma}-\widehat{\gamma}\,\delta\,(1+\alpha(\alpha\,q))\,(q^n+q^m)+(1+(\alpha\,q)^2\,q^{n+m}\,\widehat{\delta^2}\,\}.
\end{array}
\end{equation}
\end{thm}
\begin{proof}
(\ref{29}): By using (\ref{12}) and (\ref{25}) we get,
\begin{equation*}
\begin{array}{lr}
\mathbb{BF}_n(\alpha;q)\,\mathbb{BF}_m(\alpha;q)+\mathbb{BF}_{n+1}(\alpha;q)\,\mathbb{BF}_{m+1}(\alpha;q)\\
\\
=(\frac{\alpha^n\,\widehat{\gamma}-(\alpha\,q)^n\,\widehat{\delta}}{\alpha-\alpha\,q})\,(\frac{\alpha^m\,\widehat{\gamma}-(\alpha\,q)^m\,\widehat{\delta}}{\alpha-\alpha\,q})+(\frac{\alpha^{n+1}\,\widehat{\gamma}-(\alpha\,q)^{n+1}\,\widehat{\delta}}{\alpha-\alpha\,q})\,(\frac{\alpha^{m+1}\,\widehat{\gamma}-(\alpha\,q)^{m+1}\,\widehat{\delta}}{\alpha-\alpha\,q})\\
\\
=\frac{\alpha^{n+m}}{(\alpha-\alpha\,q)^2}\,\{\,(\widehat{\gamma}-q^n\,\widehat{\delta})(\widehat{\gamma}-q^m\,\widehat{\delta})\}+\frac{\alpha^{n+m+2}}{(\alpha-\alpha\,q)^2}\,\{\,\widehat{\gamma}-q^{n+1}\,\widehat{\delta})(\widehat{\gamma}-q^{m+1}\,\widehat{\delta})\}\\
\\
=\frac{\alpha^{n+m}}{(\alpha-\alpha\,q)^2}\,\{\,(1+\alpha^2)\,{\widehat{\gamma}}^2-\widehat{\gamma}\,\widehat{\delta}\,(1+\alpha(\alpha\,q))\,(q^n+q^m)+(1+(\alpha\,q)^2)\,q^{n+m}\,{\widehat{\delta}}^2\}. 
\end{array}
\end{equation*} 
where  $\widehat{\gamma }\,\widehat{\delta}=\widehat{\delta}\,\widehat{\gamma}$.
\end{proof}

\medskip

\begin{thm} \textbf{(d'Ocagne's identity)} \\ 
For $n,m\ge 0$ the d'Ocagne's identity for  the $q$-Fibonacci bicomplex quaternions $\mathbb{BF}_n(\alpha;q)$ and $\mathbb{BF}_m(\alpha;q)$ is given by
\begin{equation}\label{30}
\begin{array}{lr}
\mathbb{BF}_m(\alpha;q)\,\mathbb{BF}_{n+1}(\alpha;q)-\mathbb{BF}_{m+1}(\alpha;q)\,\mathbb{BF}_n(\alpha;q)=&\frac{\alpha^{n+m-1}(q^n-q^m)\,\widehat{\gamma}\,\widehat{\delta}}{(1-q)}. 
\end{array}
\end{equation}
\end{thm}
\begin{proof}
(\ref{30}): By using (\ref{12}) and (\ref{25}) we get,
\begin{equation*}
\begin{array}{lr}
\mathbb{BF}_m(\alpha;q)\,\mathbb{BF}_{n+1}(\alpha;q)-\mathbb{BF}_{m+1}(\alpha;q)\,\mathbb{BF}_n(\alpha;q)\\
\\
=(\frac{\alpha^m\,\widehat{\gamma}-(\alpha\,q)^m\,\widehat{\delta}}{\alpha-\alpha\,q})\,(\frac{\alpha^{n+1}\,\widehat{\gamma}-(\alpha\,q)^{n+1}\,\widehat{\delta}}{\alpha-\alpha\,q})-(\frac{\alpha^{m+1}\,\widehat{\gamma}-(\alpha\,q)^{m+1}\,\widehat{\delta}}{\alpha-\alpha\,q})\,(\frac{\alpha^{n}\,\widehat{\gamma}-(\alpha\,q)^{n}\,\widehat{\delta}}{\alpha-\alpha\,q})\\
\\
=\frac{\alpha^{n+m+1}}{(\alpha-\alpha\,q)^2}\,\{(1-q)\,(q^n-q^m)\,\widehat{\gamma}\,\widehat{\delta}\,\}\\
\\
=\frac{\alpha^{n+m-1}(q^n-q^m)\,\widehat{\gamma}\,\widehat{\delta}}{(1-q)}.
\end{array}
\end{equation*}
Here, $\widehat{\gamma}\,\widehat{\delta}=\widehat{\delta}\,\widehat{\gamma}$ is used.
\end{proof}

\medskip

\begin{thm} \textbf{(Cassini Identity)} \\
Let $\mathbb{BF}_n(\alpha;q)$ be the $q$-Fibonacci bicomplex quaternion. For $n\ge 1$, Cassini's identity for $\mathbb{BF}_n(\alpha;q)$ is as follows: 
\begin{equation}\label{31}
\mathbb{BF}_{n+1}(\alpha;q)\,\mathbb{BF}_{n-1}(\alpha;q)-\mathbb{BF}_n(\alpha;q)^2=\frac{\alpha^{2n-2}\,q^n\,(1-q^{-1})\,\widehat{\gamma}\,\widehat{\delta}}{(1-q)} .  
\end{equation}
\end{thm}
\begin{proof}
(\ref{31}): By using (\ref{12}) and (\ref{25}) we get
\begin{equation*}
\begin{array}{rl}
\mathbb{BF}_{n+1}(\alpha;q)\,\mathbb{BF}_{n-1}(\alpha;q)-\mathbb{BF}_n(\alpha;q)^2=&(\frac{\alpha^{n+1}\,\widehat{\gamma}-(\alpha\,q)^{n+1}\,\widehat{\delta}}{\alpha-\alpha\,q})\,(\frac{\alpha^{n-1}\,\widehat{\gamma}-(\alpha\,q)^{n-1}\,\widehat{\delta}}{\alpha-\alpha\,q})\\
&-(\frac{\alpha^n\,\widehat{\gamma}-(\alpha\,q)^n\,\widehat{\delta}}{(\alpha-\alpha\,q)})^2 \\
=&\frac{\alpha^{2n}\,q^n\,(1-q)(1-q^{-1})\,\widehat{\gamma}\,\widehat{\delta}}{(\alpha-\alpha\,q)^2} \\
=&\frac{\alpha^{2n-2}\,q^n\,(1-q^{-1})\,\widehat{\gamma}\,\widehat{\delta}}{(1-q)}\, .  
\end{array}
\end{equation*} 
Here, $\widehat{\gamma}\,\widehat{\delta}=\widehat{\delta}\,\widehat{\gamma}$ is used.  
\end{proof}

\medskip

\begin{thm} \textbf{(Catalan's Identity)} \\
Let $\mathbb{BF}_n(\alpha;q)$ be the $q$-Fibonacci bicomplex quaternion. For $n\ge 1$, Catalan's identity for $\mathbb{BF}_n(\alpha;q)$ is as follows: 
\begin{equation}\label{32}
\mathbb{BF}_{n+r}(\alpha;q)\,\mathbb{BF}_{n-r}(\alpha;q)-\mathbb{BF}_n(\alpha;q)^2=\frac{\alpha^{2n-2}\,q^n\,(1-q^r)(1-q^{-r})\,\widehat{\gamma}\,\widehat{\delta}}{(1-q)^2}\, .
\end{equation}
\end{thm}
\begin{proof}
(\ref{32}): By using (\ref{12}) and (\ref{25}) we get 
\begin{equation*}
\begin{array}{rl}
\mathbb{BF}_{n+r}(\alpha;q)\,\mathbb{BF}_{n-r}(\alpha;q)-\mathbb{BF}_n(\alpha;q)^2=&(\frac{\alpha^{n+r}\,\widehat{\gamma}-(\alpha\,q)^{n+r}\,\widehat{\delta}}{\alpha-\alpha\,q})\,(\frac{\alpha^{n-r}\,\widehat{\gamma}-(\alpha\,q)^{n-r}\,\widehat{\delta}}{\alpha-\alpha\,q})\\
&-(\frac{\alpha^n\,\widehat{\gamma}-(\alpha\,q)^n\,\widehat{\delta}}{(\alpha-\alpha\,q)^2} \\
=&\frac{-\alpha^{2n}\,q^{n-r}\,\widehat{\gamma}\,\widehat{\delta}-\alpha^{2n}\,q^{n+r}\,\widehat{\gamma}\,\widehat{\delta}+2\,\alpha^{2n}\,q^{n}\,\widehat{\gamma}\,\widehat{\delta}}{(\alpha-\alpha\,q)^2} \\
=&-\frac{\alpha^{2n}\,q^{n}\,\widehat{\gamma}\,\widehat{\delta}\,[\,(q^{-r}-1)+(q^r-1)\,]}{(\alpha-\alpha\,q)^2 } \\
=&\frac{\alpha^{2n-2}\,q^{n}\,\widehat{\gamma}\,\widehat{\delta}\,(1-q^{-r})(1-q^r)}{(1-q)^2}
\end{array}
\end{equation*}
Here, $\widehat{\gamma}\,\widehat{\delta}=\widehat{\delta}\,\widehat{\gamma}$ is used.
\end{proof}

\section{Conclusion} 
In this paper, algebraic and analytic properties of the $q$-Fibonacci bicomplex quaternions are investigated.\\

\bibliographystyle{elsarticle-num} 



\bibliography{bibfile}

\end{document}